\documentclass[draft]{amsart}

\usepackage{amsthm}
\usepackage[leqno]{amsmath}
\usepackage{latexsym,amsfonts,amssymb}
\usepackage[all]{xy} \SelectTips{eu}{} \SilentMatrices
\usepackage[urlcolor=blue]{hyperref}

\newcommand{\numberseries}{\mdseries}   

\newlength{\thmtopspace}                
\newlength{\thmbotspace}                
\newlength{\thmheadspace}               
\newlength{\thmindent}                  

\renewcommand{\subparagraph}{\vspace{\thmbotspace}}

\newtheoremstyle{bfupright head,slanted body}
                {\thmtopspace}{\thmbotspace}
                {\slshape}{\thmindent}{\bfseries}{.}{\thmheadspace}
                {{\numberseries \thmnumber{\bf #2 }}\thmnote{#3}}

\newtheoremstyle{bfupright head,upright body}
                {\thmtopspace}{\thmbotspace}
                {\upshape}{\thmindent}{\bfseries}{.}{\thmheadspace}
                {{\numberseries \thmnumber{\bf #2 }}\thmnote{#3}}

\newtheoremstyle{bfit head,upright body}
                {\thmtopspace}{\thmbotspace}
                {\upshape}{\thmindent}{\upshape}{.}{\thmheadspace}
                {{\numberseries\thmnumber{\bf #2 }}
                {\bfseries\itshape\thmnote{\negthickspace#3}}}

\newtheoremstyle{it head,upright body}
                {\thmtopspace}{\thmbotspace}
                {\upshape}{\thmindent}{\upshape}{.}{\thmheadspace}
                {{\numberseries\thmnumber{\bf #2 }}
                {\itshape\thmnote{\negthickspace#3}}}

\newtheoremstyle{fixed bf head,slanted body}
                {\thmtopspace}{\thmbotspace}{\slshape}
                {\thmindent}{\bfseries}{.}{\thmheadspace}
                {{\numberseries \thmnumber{\bf  #2 }}\thmname{#1}\thmnote{ (#3)}}

\newtheoremstyle{fixed bf head,upright body}
                {\thmtopspace}{\thmbotspace}{\upshape}
                {\thmindent}{\bfseries}{.}{\thmheadspace}
                {{\numberseries \thmnumber{\bf #2 }}\thmname{#1}\thmnote{ (#3)}}

\newtheoremstyle{fixed bfit head,upright body}
                {\thmtopspace}{\thmbotspace}{\upshape}
                {\thmindent}{\bfseries\itshape}{.}{\thmheadspace}
                {{\numberseries \thmnumber{\bf#2 }}\thmname{#1}\thmnote{ (#3)}}

\newtheoremstyle{sc head,small body}
                {\thmtopspace}{\thmbotspace}
                {\small\upshape}{\thmindent}{\scshape}{.}{\thmheadspace}
                {\thmname{#1}}

\newtheoremstyle{numbered paragraph}
                {\thmtopspace}{\thmbotspace}{\upshape}
                {\thmindent}{\upshape}{}{0pt}
                {{\numberseries \thmnumber{\bf #2 }}}

\newtheoremstyle{unnumbered paragraph}
                {\thmtopspace}{\thmbotspace}{\upshape}
                {\parindent}{\upshape}{}{0pt}

\theoremstyle{bfupright head,slanted body}
\newtheorem{res}{}[section]             \newtheorem*{res*}{}

\theoremstyle{bfit head,upright body}
                 \newtheorem*{com*}{}

\theoremstyle{bfupright head,upright body}
\newtheorem{bfhpg}[res]{}               \newtheorem*{bfhpg*}{}

\theoremstyle{it head,upright body}
               \newtheorem*{ithpg*}{}

\theoremstyle{sc head,small body}

\theoremstyle{fixed bf head,slanted body}
\newtheorem{thm}[res]{Theorem}          \newtheorem*{thm*}{Theorem}
\newtheorem{prp}[res]{Proposition}      \newtheorem*{prp*}{Proposition}
\newtheorem{cor}[res]{Corollary}        \newtheorem*{cor*}{Corollary}
\newtheorem{lem}[res]{Lemma}            \newtheorem*{lem*}{Lemma}

\theoremstyle{fixed bf head,upright body}
\newtheorem{dfn}[res]{Definition}       \newtheorem*{dfn*}{Definition}
\newtheorem{con}[res]{Construction}     \newtheorem*{con*}{Construction}
      \newtheorem*{obs*}{Observation}
\newtheorem{rmk}[res]{Remark}           \newtheorem*{rmk*}{Remark}
          \newtheorem*{exa*}{Example}
         \newtheorem*{exe*}{Exercise}
            \newtheorem{stp*}{Setup}

\theoremstyle{numbered paragraph}

\theoremstyle{unnumbered paragraph}
\newtheorem{ipg*}{}

\newlength{\thmlistleft}        
\newlength{\thmlistright}       
\newlength{\thmlistpartopsep}   
\newlength{\thmlisttopsep}      
\newlength{\thmlistparsep}      
\newlength{\thmlistitemsep}     

\newcounter{eqc}
  {\end{list}}%

\newcounter{prt}
  {\end{list}}%

\newcounter{rqm}
  {\end{list}}%

  {\end{list}}%

  {\end{list}}%

\newcounter{exercise}
  {\end{list}}%

\newenvironment{prf*}[1][Proof]{%
  \begin{proof}[\it #1]
    \setcounter{equation}{0}
    \renewcommand{\theequation}{\arabic{equation}}}
  {\end{proof}
}

\newcommand{\pgref}[1]{(\ref{#1})}

\renewcommand{\eqref}[1]{\pgref{eq:#1}}

\numberwithin{equation}{res}
\newcounter{marcom}

\renewcommand{\ge}{\geqslant}
\newcommand{\onto}{\twoheadrightarrow}

\newcommand{\m}{\mathfrak{m}}

\newcommand{\p}{\mathfrak{p}}
\newcommand{\Hom}[3][R]{\operatorname{Hom}_{#1}(#2,#3)}
\newcommand{\PHom}{\operatorname{Hom}}
\newcommand{\Ext}[4][R]{\operatorname{Ext}_{#1}^{#2}(#3,#4)}

\newcommand{\ann}[1]{\operatorname{ann}(#1)}
\newcommand{\depth}{\operatorname{depth}}
\newcommand{\grade}[1]{\operatorname{grade}( #1)}
\newcommand{\height}[1]{\operatorname{ht}( #1)}

\renewcommand{\Gamma}{\textrm{C}}

\newcommand{\koszul}{\textrm{K}}
\newcommand{\homology}{\textrm{H}}
\newcommand{\tot}{\textrm{Tot}}
\def\urltilda{\kern -.15em\lower .7ex\hbox{\~{}}\kern .04em}
\def\widebardisplay#1{%
  \setbox0=\hbox{$\displaystyle #1$}
  \dimen0=\wd0%
  \advance\dimen0 by -3.2pt
  \vbox{%
    \nointerlineskip%
    \moveright 1.2pt 
    \vbox{\hrule width \dimen0}%
    \nointerlineskip%
    \kern 1.25pt
    \box0%
    }%
  }

\def\widebartext#1{%
  \setbox0=\hbox{$#1$}
  \dimen0=\wd0%
  \advance\dimen0 by -3.2pt
  \vbox{%
    \nointerlineskip%
    \moveright 1.2pt 
    \vbox{\hrule width \dimen0}%
    \nointerlineskip%
    \kern 1.25pt
    \box0%
    }%
  }

\def\widebarscript#1{%
  \setbox0=\hbox{$\scriptstyle #1$}
  \dimen0=\wd0%
  \advance\dimen0 by -2pt
  \vbox{%
    \nointerlineskip%
    \moveright 1pt 
    \vbox{\hrule width \dimen0}%
    \nointerlineskip%
    \kern .8pt
    \box0%
    }%
  }

\def\widebarscriptscript#1{%
  \setbox0=\hbox{$\scriptscriptstyle #1$}
  \dimen0=\wd0%
  \advance\dimen0 by -1pt
  \vbox{%
    \nointerlineskip%
    \moveright .5pt 
    \vbox{\hrule width \dimen0}%
    \nointerlineskip%
    \kern .6pt
    \box0%
    }%
  }

\begin{document}

\allowdisplaybreaks[2]

\title{Duality for Koszul Homology over Gorenstein Rings}

\author[C.\, Miller  ]{Claudia Miller}

\address{C.M.\newline\hspace*{1em} Mathematics Department, Syracuse University, Syracuse, NY
  13244, U.S.A.}

\email{clamille@syr.edu}

\urladdr{http://www.phy.syr.edu/$\sim$clamille/}

\author[H. Rahmati]{Hamidreza Rahmati}

\address{H.R.\newline\hspace*{1em} Mathematics Department, Syracuse University, Syracuse, NY
  13244, U.S.A.}

\email{hrahmati@syr.edu}

\author[J. Striuli]{Janet Striuli}

\address{J.S.\newline\hspace*{1em} Department of Math.\ and C.S.,
  Fairfield University, Fairfield, CT~06824,~U.S.A.}

\email{jstriuli@fairfield.edu}

\urladdr{http://www.faculty.fairfield.edu/jstriuli}

\thanks{This research was partly supported by  NSF grant DMS\,0901427 (J.S.), and  NSF grant DMS\, 100334 (C.M.).}

\date{\today}




\begin{abstract}
We study Koszul homology over Gorenstein rings. If an ideal is strongly Cohen-Macaulay, the Koszul homology algebra satisfies Poincar\'e duality. We prove a version of this duality which holds for all ideals and allows us to give two criteria for an ideal to be strongly Cohen-Macaulay. The first can be compared to a result of Hartshorne and Ogus; the second is a generalization of a result of Herzog, Simis, and Vasconcelos using sliding depth. 
\end{abstract}

\maketitle

\section*{Introduction}
\label{sec:intro}
We study duality properties of Koszul homology and their
implications. Let $I$ be an ideal of a commutative noetherian ring $R$
and $\homology_i(I)$ be its $i$th Koszul homology module. The differential
graded algebra structure on the Koszul complex induces homomorphisms
\begin{align*}
\homology_{l-g-i}(I)  \xrightarrow{\varphi_i}
\Hom{\homology_i(I)}{\homology_{l-g}(I)} ,
\end{align*}
for all $i$, where $\ell$ is the minimal number of
generators of $I$ and $g$ is its grade.

If every map $\varphi_i$ is an isomorphism then the Koszul homology
algebra is said to be Poincar\'e, equivalently the ideal is said to
satisfy Poincar\'e duality.  In \cite{H}, Herzog proves that the Koszul
homology algebra is Poincar\'e when the ring is Gorenstein and the
ideal is strongly Cohen-Macaulay, that is when all of its
non-vanishing Koszul homologies are Cohen-Macaulay modules.  A proof
of this fact can also be found in \cite{G}. Herzog's result extends
the work of Avramov and Golod \cite{AG}, where the authors prove that
the maximal ideal of a local ring satisfies Poincar\'e duality if and only if the ring is Gorenstein.

We first give a version of this kind of duality which holds for all
ideals of a Gorenstein ring; see Corollary~\ref{cor:compare}.

\vspace{1mm}

\noindent {\bf Theorem A.}
{\em{
Let $R$ be a Gorenstein ring and $I$ be an ideal of grade $g$ that is minimally generated
by $\ell$ elements. Then there are isomorphisms
\begin{align*}
  \Hom{\Hom{\homology_i(I)}{\homology_{\ell-g}(I)}}{{\homology_{\ell-g}(I)}}
  &\cong \Hom{\homology_{\ell-g-i}(I)}{\homology_{\ell-g}(I)}
\end{align*}
for all $i\geq 0$.
}}

\vspace{1mm}

We also prove a generalized version of this theorem for Cohen-Macaulay
rings with canonical module; see Theorem
\ref{thm:genduality-module}. In \cite{Ch}, Chardin proves these
results in the graded setting using a spectral sequence.  Our proof,
which is the content of Section 1, employs a different spectral sequence. 

In Sections 2 and 3, further analysis of this spectral sequence allows us 
to give several criteria for an ideal to be strongly Cohen-Macaulay 
in the presence of conditions ensuring that duality holds; 
see Proposition~\ref{prp:S2} and Remark~\ref{rmk:equivalenttoS2}. 

Strongly Cohen-Macaulay ideals were formally
introduced in \cite{Hu} and since then have been the subject of
intense study.  The interest in these ideals is justified by the
nice geometric properties of the schemes they define;  
under proper assumptions the Rees algebra and the symmetric algebra 
of these ideals are isomorphic and Cohen-Macaulay; see \cite{SV} and \cite{Va-80}.  
A large class  of strongly Cohen-Macaulay ideals is given
by ideals that are in the linkage class of a complete intersection, as shown by Huneke in \cite{hu82}. 
Over a Gorenstein ring any ideal whose minimal number of generators
$\mu(I)$ is at most $\grade I +2$ is strongly Cohen-Macaulay; see
\cite{Ku}, \cite{AH}.

In Section 2 we give our first criterion for an ideal to be strongly
Cohen-Macaulay; see Theorem~\ref{thm:dimR/2}.

\vspace{1mm}

\noindent {\bf Theorem B.}
{\em{
  Let $R$ be a local Gorenstein ring and $I$ be an ideal of $R$.  Let
  $h$ be an integer such that $h \geq \max \{ 2, \frac{1}{2} \dim
  (R/I) \}$.  If, for all $i\geq 0$, the $R/I$-module $\homology_i(I)$ 
  satisfies Serre's condition $S_h$, then $I$ is strongly
  Cohen-Macaulay.
}}

\vspace{1mm}

This theorem extends to all Koszul homology modules the following
criterion of Hartshorne and Ogus \cite{HO}: Let $R$ be a Gorenstein
local ring and $I$ be an ideal of $R$ such that $R/I$ is factorial. If
$R/I$ satisfies $S_2$ and the inequality $\depth (R/I)_{\mathfrak{p}} \geq \frac{1}{2}(\dim
(R/I)_{\mathfrak{p}})+1$ holds for all prime ideals $\mathfrak{p}$ of $R/I$ of height at least 2, 
then $R/I$ is Cohen-Macaulay.  Notice that by considering all
the Koszul homologies, one can lower the depth required
and remove the dependence on factoriality.

The proof of Theorem B comes from a careful analysis of the spectral
sequence employed in the proof of Theorem~\ref{thm:genduality},
together with a generalization of a result of Huneke \cite{Hu}, 
namely Proposition~\ref{prp:S2}, which is then used in the rest of the paper. 
Moreover, we use a result from  \cite{HO};
see the discussion in Remark~\ref{rmk:HO}.   

In Section 3 we give another criterion for an ideal to be strongly
Cohen-Macaulay. In \cite{HVV} Herzog, Vasconcelos and Villareal introduce the notion
of sliding depth which corresponds to the case $h=0$ of the later
generalization $h$-sliding depth $SD_h$ in \cite{HSV}: an ideal $I$ satisfies
$SD_h$ if
\[
\depth \homology_i(I) \geq \min\{\dim R-\grade I, \dim R
-\mu (I) +i +h\}.
\]
In \cite{HVV} the authors prove that if a Cohen-Macaulay ideal $I$ satisfies 
the inequality $\mu(I_{\p}) \leq \max\{\height{I},\height{\p} -1\}$ for every prime ideal $\p$
containing $I$ and satisfies sliding depth, then $I$ is strongly Cohen-Macaulay.  
Using the spectral sequence from the proof of Theorem~\ref{thm:genduality}, we
recover and extend this result in Theorem~\ref{thm:SDh} by weakening the condition
on $\mu(I)$ while strengthening sliding depth to $h$-sliding depth
$SD_h$.  Among the many corollaries we obtain, the most interesting is
perhaps the following; see Corollaries~\ref{SDhj=0} and \ref{SDhj=1}.

\vspace{1mm}

\noindent {\bf Theorem C.}
{\em{ Let $R$ be a local Gorenstein ring, and let $I$ be an ideal of $R$ such that
$R/I$ satisfies Serre's condition $S_2$. Suppose that $I$  satisfies $SD_1$ and
one of the following two conditions holds
\begin{itemize}
\item[(a)] $\mu(I_{\p}) \leq \height{\p}$ for all prime ideals $\p \supseteq I$
or
\item[(a$^\prime$)\!] $\mu(I_{\p}) \leq \height{\p}+1$ for all prime ideals $\p \supseteq I$
and $\homology_1(I)$ satisfies $S_2$.
\end{itemize}
\noindent Then $I$ is strongly Cohen-Macaulay.
}}
\vspace{1mm}

In Section 4 we present some results on the relationship between the Koszul homology modules for low dimensional ideals. 

We finish the introduction by settling some notation. Given a sequence ${\bf y}$ of
$\ell$ elements in $R$, we denote by $\koszul({\bf y})$ the Koszul
complex on the elements $\bf{y}$ and by $\homology_i ({\bf y})$ its
$i$th homology module. The $i$th cohomology module of the complex
$\Hom{\koszul ({\bf y})}{R}$ is the Koszul cohomology, and it is
denoted by $\homology^{i}({\bf y})$.   For an $R$-module
$M$, its Koszul complex and its Koszul homology and cohomology modules
are denoted by $\koszul({\bf y};M)$, $\homology_i({\bf y};M)$ and
$\homology^i({\bf y};M)$, respectively. The Koszul complex of an ideal
$I$ is computed on a minimal set of generators, and it is denote by
$\koszul(I)$. Its homology and cohomology modules are denoted by
$\homology_i(I)$ and by $\homology^i(I)$. Moreover, we will often use the
isomorphism $\homology^{i}({\bf y};M)\cong \homology_{\ell -i}({\bf
  y};M)$. If the sequence ${\bf y}$ is of
length $\ell$ then the only non-zero homology modules are
$\homology_i({\bf y})$ for $i=0, \dots, \ell -g$,
where $g$ is the grade of the ideal generated by ${\bf y}$, 
and each of these has dimension equal to the dimension of $R/I$ 
since they have the same support over $R$.

In the first half of the paper, we give our results in terms of Koszul homology on a sequence of
elements rather than the Koszul homology of an ideal. 
We do this because in the proofs of our applications we often localize and lose
minimality. However, in Sections 3 and 4 the results are given for ideals since, historically, that is the main case of interest.

Throughout the paper we  use Serre's conditions $S_n$: 
An $R$-module $M$ is said to satisfy $S_n$ if the inequality 
$\depth M_{\mathfrak{p}} \geq \min\{n,\dim{R_{\mathfrak{p}}}\}$ holds for every prime ideal
$\mathfrak{p}$ in $ R$.  We also use the following fact:

\vspace{3mm}

\noindent {\bf Fact.}
{\em{Let $M$ and $N$ be  $R$-modules which satisfy $S_2$ and $S_1$,
respectively. Assume that the $R$-homomorphism $\varphi\colon M\to N$ gives an
isomorphism when localized at any prime ideal $\mathfrak{p}$ such that
$\height {\mathfrak {p}} \leq 1$. Then $\varphi$ is an isomorphism.
}}
\vspace{3mm}

\section{Spectral sequence and generalized duality result} 
\label{sec:PD}

The aim of this section is to give a duality modeled on Poincar\'e
duality that holds for all ideals of a Gorenstein ring.  We also prove
a more general version for Cohen-Macaulay rings with a canonical
module $\omega_R$.  In \cite[Lemma 5.7]{Ch}, Chardin proves this in
the graded setting; his proof invokes a different spectral sequence
from the one that we use, which we describe below.

\begin{con}
\label{rmk:spectralsequence}
Let ${\bf y}$ be a sequence of $\ell$ elements generating an ideal
$({\bf y})$ of grade $g$.  Let ${\textrm J}$ be a minimal injective
resolution of $R$, and let $\koszul$ denote the Koszul complex
$\koszul({\bf y})$.  Consider the double complex
$\Gamma=\Hom{\koszul}{{\textrm J}}$ with
$\Gamma^{p,q}=\Hom{\koszul_p}{{\textrm J}^{q}}$ and its two associated
spectral sequences. Taking homology first in the vertical $q$
direction, one obtains a collapsing spectral sequence, yielding 
\begin{align*}
  \homology^n(\tot \Gamma)= \homology_n(\Hom{\koszul }{R}) \cong
  \homology^n({\bf y})\cong\homology_{\ell-n}({\bf y}).
\end{align*}
The second convergent spectral sequence, obtained by taking homology
first in the horizontal $p$ direction, is therefore of the form
\begin{eqnarray*}
  \label{e1}
  E_{2}^{p,q}=\Ext{q}{\homology_p({\bf y})}{R} \Rightarrow \homology_{\ell-(p+q)}({\bf y}).
\end{eqnarray*}
Note that $E_{2}^{p,q}=0$ for $q<g$, since the ideal generated
by ${\bf y}$ has grade $g$ and annihilates $\homology_i({\bf y})$.
Furthermore, when $R$ is Gorenstein of dimension $d$, one has that
$E_{2}^{p,q}=0$ for $q>d$.
\end{con}

We now give a duality theorem for all ideals of a Gorenstein ring.

\begin{thm}
\label{thm:genduality}
Let $R$ be a Gorenstein  ring and ${\bf y}$ be a sequence of 
$\ell$ elements that generate an ideal of grade $g$.  
Then there is an
isomorphism
\begin{align*}
\Ext{g}{\Ext{g}{\homology_i({\bf y})}{R}}{R}  &\cong  \Ext{g}{\homology_{\ell-g-i}({\bf y})}{R}
\end{align*}
for every $i\geq 0$.
\end{thm}

The idea is that one more application of the functor $\Ext{g}{-}{\omega_R}$,
itself a duality on Cohen-Macaulay modules of dimension $d-g$, removes
the lower-dimensional obstructions to it being a duality on Koszul homology.

Before we prove this theorem, we restate it in a way that enables one to
compare it to classical Poincar\'e duality. To do so, we employ some
well-known isomorphisms.

\begin{rmk}
\label{rmk:ExtgHom}
Let $R$ be a Cohen-Macaulay ring and ${\bf y}$ a sequence of $\ell$
elements that generate an ideal $I$ of grade $g$.  It is
well-known that the top non-vanishing Koszul homology satisfies
$\homology_{\ell-g}({\bf y}) \cong \Ext{g}{R/I}{R}$. Let $M$ be an
$R/I$-module.  Choosing any $R$-regular sequence ${\bf x}$ of length
$g$ in the ideal $I$, one therefore obtains natural isomorphisms (the
third one being adjunction) for all $i$,
\begin{align*}
\Hom{M}{\homology_{l-g}({\bf y})} &\cong \Hom{M}{\Ext{g}{R/I}{R}}  \\
&\cong \Hom{M}{\Hom{R/I}{R/({\bf x})}}  \\
&\cong \Hom{M}{R/({\bf x})}  \\
&\cong \Ext{g}{M}{R}.
\end{align*}
\end{rmk}

The identifications in the remark now yield the 
following version of Theorem~\ref{thm:genduality}. 

\begin{cor}
\label{cor:compare}
Let $R$ be a Gorenstein  ring and ${\bf y}$ a set of
$\ell$ elements that generate an ideal of grade $g$.  
Then there is an isomorphism
\begin{align*}
\Hom{\Hom{\homology_i({\bf y})}{\homology_{\ell-g}({\bf y})}}{{\homology_{\ell-g}({\bf y})}}   &\cong
\Hom{\homology_{\ell-g-i}({\bf y})}{\homology_{\ell-g}({\bf y})}
\end{align*}
for all $i\geq 0$.
\end{cor}


Now we give the proof of the theorem.

\begin{prf*}[Proof of Theorem~\ref{thm:genduality}]
Set $d$ to be the dimension of $R$ and $I$ the ideal generated by
the seqence ${\bf y}$.  Consider the spectral sequence from
Construction~\ref{rmk:spectralsequence}.  The edge homomorphisms 
provide maps
 \[
  \psi_i\colon \homology_{\ell-g-i}({\bf y}) \to
  \Ext{g}{\homology_i({\bf y})}{R}
\]
  and hence maps
\begin{align*}
  \Ext{g}{\Ext{g}{\homology_i({\bf y})}{R}}{R}
  &\xrightarrow{\Ext{g}{\psi_i}{R}} \Ext{g}{\homology_{\ell-g-i}({\bf
      y})}{R}.
\end{align*}
Both modules above satisfy Serre's condition $S_2$ as $R/I$-modules,
for they can be rewritten in the form $\PHom_{R/I}(-,R/({\bf x}))$,
where ${\bf x}$ is any $R$-regular sequence of length $g$ in $I$.
Thus, by the fact from the introduction, to show that the
map $\Ext{g}{\psi_i}{R}$ is an isomorphism for each $i$, it suffices
to prove that it is an isomorphism in codimension 1 over the ring
$R/I$.  Localizing at a prime ideal of $R$ whose image in $R/I$ has
height at most 1, one sees that it remains to prove that
$\Ext{g}{\psi_i}{R}$ is an isomorphism for the case that $\dim R/I
\leq 1$, that is, $d-g \leq 1$.

If $\dim R/I =0$, that is, $d=g$, the spectral sequence collapses to
a single row at $q=g$, and so the edge homomorphism $\psi_i$ is
already an isomorphism for all $i\geq 0$.

If $\dim R/I = 1$, that is, $d=g+1$, the spectral sequence has at most
two nonzero rows and thus satisfies $E_2=E_\infty$ and yields particularly short filtrations of
$\homology^n(\tot \Gamma)$, namely given by the exact sequences
involving the edge homomorphisms 
\begin{align*}
\label{e2}
&0 \to
\Ext{d}{\homology_i({\bf y})}{R}  \to
\homology_{\ell-g-i}({\bf y})  \xrightarrow{\psi_i}
\Ext{g}{\homology_i({\bf y})}{R} \to
0.
\end{align*}
Set $L=\Ext{d}{\homology_i({\bf y})}{R}$, and apply $\Ext{g}{-}{R}$ to
the sequence above.  Since $R$ is Gorenstein, the module $L$ has
finite length, and so $\Ext{j}{L}{R}$ vanishes for all $j<d$, yielding
the desired statement.
\end{prf*}

\begin{rmk}
\label{rmk:edge}
The isomorphisms in Theorem~\ref{thm:genduality} are induced by the edge
homomorphism from the spectral sequence of
Construction~\ref{rmk:spectralsequence}.
\end{rmk}

We now extend Theorem~\ref{thm:genduality} to the setting of
Cohen-Macaulay rings. This generalization is also given by Chardin in
\cite{Ch} for $M=R$ in the graded case.

\begin{thm}
\label{thm:genduality-module}
Let $R$ be a Cohen-Macaulay noetherian ring with canonical module
$\omega_R$. Let $M$ be a maximal Cohen-Macaulay module, and let ${\bf
  y}$ be a sequence of elements that generate an ideal of grade $g$.
Then there is an isomorphism
\begin{align*}
\Ext{g}{\Ext{g}{\homology_i({\bf y};M)}{\omega_R}}{\omega_R}\cong \Ext{g}{\homology_{\ell-g-i}({\bf y};\Hom{M}{\omega_R})}{\omega_R}
\end{align*}
for every $i\geq 0$.
\end{thm}

\begin{proof}
  The proof is the same as that of Theorem~\ref{thm:genduality} with
  the following minor modifications.  Consider instead the double
  complex ${\textrm C}=\Hom{\koszul \otimes_R M}{\textrm J}$, where
  $\koszul$ is the Koszul complex $\koszul({\bf y})$ and ${\textrm J}$
  is an injective resolution of $\omega_R$, and the two spectral
  sequences associated to it.  Taking homology first in the vertical
  $q$ direction, one again obtains a collapsing spectral sequence as
  each $\koszul_p\otimes_R M$ is maximal Cohen-Macaulay and hence
  ${}^{I}E^{p,q}_{1} \cong \Ext{q}{\koszul_p\otimes_R M}{\omega_R}$
  vanishes for all $q>0$.  This implies that ${}^{I}E^{p,0}_{2}$ is
  the $p$th homology of the total complex ${\textrm C}$.  So, one
  can see that
\begin{align*}
\homology^n(\tot \Gamma)
&\cong \homology_n(\Hom{\koszul\otimes_R M}{\omega_R}) \\
&\cong \homology_n(\Hom{\koszul}{\Hom{M}{\omega_R}}) \\
&\cong \homology_{\ell -n}({\bf y},\Hom{M}{\omega_R}).
\end{align*}

In the other direction one obtains a convergent spectral sequence of the form
\begin{eqnarray*}
\label{e3}
{}^{II}E_{2}^{p,q}=\Ext{q}{\homology_p({\bf y};M)}{\omega_R} \Rightarrow \homology_{\ell-(p+q)}({\bf y};\Hom{M}{\omega_R}).
\end{eqnarray*}

Now the same argument as in Theorem~\ref{thm:genduality} applied to
this spectral sequence yields the desired result since $\omega_R$ has
injective dimension $d$ and is a maximal Cohen-Macaulay module, hence
the longest $\omega_R$-sequence in the ideal $({\bf y})$ has length $g$.
\end{proof}

\section{Consequences and related results}
\label{sec:coros}

In this section we first give extensions of two classical results, the
first of Herzog from \cite{H} and the second of Huneke from
\cite{Hu}, to the case of sequences of elements. We need these more
general versions for the theorems in this section and the next.

Then, in the main theorem of the section, by a careful analysis of the
spectral sequence from Construction~\ref{rmk:spectralsequence}, we
show that the Koszul homologies are Cohen-Macaulay if they satisfy Serre's
conditions $S_h$ for $h$ at least half their dimension and at least 2. 
The result extends a theorem of Hartshorne and Ogus in \cite{HO} 
from $R/I$ to all the Koszul homology modules. 

We begin by proving a version of the result of Herzog from \cite{H} for arbitrary sequences of elements. 

\begin{prp}\label{prp:herzog}
  Let $R$ be a Gorenstein ring and $h$ be a nonnegative integer. Let ${\bf y}$ be
  a sequence of $\ell$ elements that generate an ideal of
  grade $g$. If the module $\homology _i({\bf y})$ is  Cohen-Macaulay for all
  $i\leq h$, then there is an isomorphism
\[
\Ext{g}{\homology_i({\bf y})}{R}\cong\homology_{\ell-g-i}({\bf y})
{\textrm{ for each }}
i\leq h+1.
\]
In particular, $\homology_{\ell-g-i}({\bf y})$ is Cohen-Macaulay for every $i\leq
h$. Moreover, the following  hold
\begin{align*}
\ann {\homology_i({\bf y})}&=\ann{ \homology_{\ell-g-i}({\bf y})}  {\textrm{ for }} 0\leq i \leq h, \, \textrm{and} \\
\ann {\homology_{h+1}({\bf y})}&\subseteq \ann{ \homology_{\ell-g-(h+1)}({\bf y})}.
\end{align*}
\end{prp}

\begin{proof}
  Set $\homology_i=\homology_i({\bf y})$ for all $i \geq 0$. Consider
  the spectral sequence
\begin{align*}
E_{2}^{p,q}=\Ext{q}{\homology_p}{R} \Rightarrow \homology_{\ell-(p+q)}
\end{align*}
from Construction~\ref{rmk:spectralsequence}.
As before, one has that $E_{2}^{p,q}=0$ for all $q<g$, but now the Cohen-Macaulay
hypothesis implies that $E_{2}^{p,q}=0$ when $p\leq h$ and $q>g$ as well.
So one has that
  $E^{p,g}_{2}=E^{p,g}_{3}=\dots=E_{\infty}^{p,g}$
  for all $p\leq h+1$. Therefore, by renaming the index $p$ by $i$ one gets
  \[
  E_{\infty}^{p,g}=\Ext{g}{\homology_i}{R} \cong
  \homology_{g+i}(\tot {\textrm C})=\homology^{g+i}
  \cong\homology_{\ell-g-i},
\]
for $i=0, \dots, h+1$, and so the first part of thesis holds.  Since
$R$ is Gorenstein and $\homology_i$ is Cohen-Macaulay for $i \leq
h$, the module $\Ext{g}{\homology_i}{R}$ is Cohen-Macaulay for each
$i \leq h$. For the remaining assertions, notice that for every $i \leq h+1$ one has
\[
\ann {\homology_i} \subseteq \ann {\Ext{g}{\homology_i}{R}}=\ann
{\homology_{\ell-g-i}}.
\]
For each $i \leq h$ the other inclusion follows from the following
isomorphisms
\[
\homology_{i} \cong \Ext{g}{\Ext{g}{\homology_i}{R}}{R} \cong
\Ext{g}{\homology_{\ell-g-i}}{R},
\]
where the first one is  \cite[Theorem 3.3.10]{bruns-herzog}.
\end{proof}

\begin{rmk}\label{rmk:edgeiso}The isomorphisms in the previous
  proposition are given by the edge homomorphism of the spectral sequence
  from Construction~\ref{rmk:spectralsequence}.
\end{rmk}

A surprising fact is that the Koszul homology algebra is Poincar\'e under the much
weaker assumption of reflexivity on the Koszul homology modules,
rather than the full hypothesis of Cohen-Macaulayness. This fact was
noted by Huneke in \cite[Prop~2.7]{Hu} for ideals. Below we use
Theorem~\ref{thm:genduality} to obtain a version, modulo the
identifications from Remark~\ref{rmk:ExtgHom}, that works for a {\em single} value of $i$ 
and any sequence of elements.

\begin{prp}
\label{prp:S2}
Let $R$ be a Gorenstein ring and ${\bf y}$ be a sequence of $\ell$
elements that generate an ideal of grade $g$.  Assume that for some
integer $i$ the Koszul homology $\homology_i({\bf y})$ satisfies
Serre's condition $S_2$ as an $R/({\bf y})$-module.  
Then the edge homomorphism from the spectral sequence of
Construction~\ref{rmk:spectralsequence} gives an isomorphism
$\homology_{i}({\bf y}) \cong \Ext{g}{\homology_{\ell-g-i}({\bf y})}{R}$. 
\end{prp}
\begin{proof}
Denote the dual $\Ext{g}{-}{R}$ by $(-)^\vee$. Note that for $R/({\bf y})$-modules, one has $(-)^\vee \cong \Hom{-}{R/({\bf x})}$, 
where ${\bf x}$ is any $R$-regular sequence of length $g$ in $({\bf y})$.
By Theorem~\ref{thm:genduality} and Remark~\ref{rmk:edge}, the dual $\psi^\vee$ 
of the edge homomorphism 
$\psi\colon\homology_{i}({\bf y}) \to  \homology_{\ell-g-i}({\bf y})^\vee$ 
is an isomorphism. 
Dualizing once more one obtains a commutative diagram 
  \begin{equation*}
  \xymatrix{
   \homology_{i}({\bf y})  \ar@{->}^{\!\!\!\!\!\!\!\!\psi}[r] \ar@{->}[d]
      & \homology_{\ell-g-i}({\bf y})^\vee \ar@{->}[d] \\
    \homology_{i}({\bf y})^{\vee\vee} \ar@{->}^{\!\!\!\!\!\!\psi^{\vee\vee}}_{\!\!\!\!\!\!\cong}[r] 
      & \  \homology_{\ell-g-i}({\bf y})^{\vee\vee\vee}
  }
  \end{equation*}
whose bottom row is an isomorphism and vertical maps are the natural maps 
to the double dual. 
The desired conclusion follows from the fact that the vertical maps are isomorphisms. 
Indeed, since the modules in the top row satisfy $S_2$ also as 
$R/({\bf x})$-modules, the left one by hypothesis and the right one by identifying it with a Hom module as above, 
they are reflexive as $R$ is Gorenstein. 
\end{proof}

\begin{rmk}
\label{rmk:allS2}
An even simpler proof of the proposition above can be given if one knows that, for all $j$, the
Koszul homology $\homology_j({\bf y})$ satisfies $S_2$ as an $R/({\bf y})$-module 
(or even just for all $j\leq \ell-g-i$) along the lines
of Huneke's original proof.  By Proposition~\ref{prp:herzog} and
Remark~\ref{rmk:edgeiso}, the edge homomorphism $\psi$ is an
isomorphism when localized at primes of height at most 2 in $R/({\bf
  x})$, where ${\bf x}$ is any $R$-regular sequence of length $g$ in
$({\bf y})$.  As in the previous proof one sees that both the domain and
target modules satisfy $S_2$ over $R/({\bf x})$, and hence $\psi$ is
an isomorphism by the fact from the introduction.
\end{rmk}

\begin{rmk}
\label{rmk:equivalenttoS2}
The converse of Proposition~\ref{prp:S2} is true as well: If the isomorphism holds for
some $i$, then $\homology_{i}({\bf y})$ satisfies $S_2$ as an $R/({\bf
  y})$-module.  Indeed, the identification of the target as
$\Ext{g}{\homology_{\ell-g-i}({\bf y})}{R} \cong
\Hom{\homology_{\ell-g-i}({\bf y})}{R/{\bf x}}$ implies that this
module satisfies $S_2$ as an $R/({\bf x})$-module and hence as an
$R/({\bf y})$ module. 

Notice that this, together with Huneke's result \cite[Prop~2.7]{Hu}, implies 
that the following conditions are equivalent:
\begin{enumerate}
\item The ideal $I$ satisfies Poincar\'e duality, i.e., the maps from the 
introduction induced by the algebra structure on the Koszul complex are isomorphisms.
\item  The Koszul homologies of $I$ satisfy  $S_2$.
\item The edge homomorphisms from the spectral sequence are isomorphisms. 
\end{enumerate}
We do not know if the duality maps from  (1) and (3) are the same, 
but clearly they are isomorphisms at the same time.
\end{rmk}

The results in the rest of the paper come from further analysis of the
spectral sequence from Construction~\ref{rmk:spectralsequence} in
situations where the edge homomorphism is an isomorphism, for example 
as in Proposition~\ref{prp:S2} above. Indeed, in such circumstances
certain strong conditions are forced on the spectral sequence, as the
following key lemma shows.

\begin{lem}
\label{arrows}
Let $R$ be a Gorenstein ring and ${\bf y}$ be a sequence of $\ell$
elements that generate an ideal of grade $g$. Consider the spectral
sequence
\begin{align*}
E_{2}^{p,q}=\Ext{q}{\homology_p({\bf y})}{R} \Rightarrow \homology_{\ell-(p+q)}({\bf y}),
\end{align*}
from Construction~\ref{rmk:spectralsequence}. Fix an $i \geq 0$. If the edge
homomorphism $\homology_i({\bf y}) \to
\Ext{g}{\homology_{\ell-g-i}({\bf y})}{R}$ is an isomorphism, then

\begin{enumerate}
\item  the differential $d_r^{\ell-g-i,g}:E_r^{\ell-g-i,g} \to E_r^{\ell-g-i-r+1,g+r}$ is zero for all $r\geq 2$, and
\item the modules $E_{\infty}^{p,q}$ are zero for any $q>g$ with $p+q=\ell-i$.
\end{enumerate}
\end{lem}

\begin{proof}
Set $\homology_i=\homology_i({\bf y})$ for all $i$.
The composition of the homomorphisms
\[
\homology_{i}\onto E_{\infty}^{\ell -g-i,g}\hookrightarrow
E_2^{\ell-g-i,g}=\Ext{g}{\homology_{\ell-g-i}}{R}
\]
is the edge homomorphism. Since it is
an isomorphism, we obtain that $E_{\infty}^{\ell -g-i,g}=
E_2^{\ell-g-i,g}=\Ext{g}{\homology_{\ell-g-i}}{R}$, which implies that 
$d_r^{\ell-g-i,g}=0$ for all $r\geq 2$.

Since the spectral sequence converges there is a
filtration $F_0 \subseteq F_1 \subseteq \dots \subseteq F_{d-g}$
such that $F_{d-g}\cong H_{i}$ and $F_{j+1}/F_j \cong
E_{\infty}^{p,q}$ with $p+q=\ell-i$ and $q=d-j$. 
As  the map $\homology_{i} \to E_{\infty}^{\ell-g-i,g}$ is an isomorphism, the filtration reduces to $F_j=0$
for all $j=0,\dots,d-g-1$ and therefore  $E_{\infty}^{p,q}=0 $ for $q>g$ with $p+q=\ell-i$.
\end{proof}

Now we are ready to state and prove an extension of a theorem of Hartshorne and Ogus from \cite{HO}.
Notice that by requiring the proper Serre's condition on all the Koszul homology
modules, one can relax them by one and also remove the factoriality hypotheses on $R/I$. 

\begin{thm}
\label{thm:dimR/2}
Let $R$ be a local Gorenstein  ring and ${\bf y}$ be a
sequence of $\ell$ elements which generate an ideal of grade $g$.  Let $h$ be
an integer such that $h \geq \max \{ 2, \frac{1}{2} \dim (R/({\bf y})
\}$.  
If, for all $i\geq 0$, the $R/({\bf y})$-module $\homology_i({\bf y})$ satisfies Serre's condition 
$S_h$, then  $\homology_i({\bf y})$ is Cohen-Macaulay for all $i=0,\dots,\ell-g$. 
\end{thm}

\begin{proof}
Set $\homology_i = \homology_i({\bf y})$ for all $i$. Note that if
$\dim R/({\bf y}) \leq 2$ or $h\geq d-g$ there is nothing to prove,
so one may assume that $\dim R/({\bf y}) \geq 3$ and $h<d-g$.  Since
the conditions in the hypotheses localize, we may thus assume by
induction that the ideal $({\bf y})$ is strongly Cohen-Macaulay on
the punctured spectrum.

We use the spectral sequence
\begin{align*}
E_{2}^{p,q}=\Ext{q}{\homology_p}{R} \Rightarrow \homology_{\ell-(p+q)}
\end{align*}
from Construction~\ref{rmk:spectralsequence}.  As before, one has that
$E_{2}^{p,q}=0$ for $q<g$, but now the hypothesis that $\homology_i$
satisfies $S_h$ implies that $E_{2}^{p,q}=0$ holds for all $q>d-h$ as well.  We
use induction on $i$ to show that $\homology_{i}$ and
$\homology_{\ell-g-i}$ are Cohen-Macaulay.

Taking as a base case for the induction $i=-1$, one sees that it holds
trivially as $\homology_{-1}=0=\homology_{\ell-g+1}$.  For the
inductive step, we suppose that $\homology_{p}$ and
$\homology_{\ell-g-p}$ are Cohen-Macaulay for all $p\leq i$, and we
show that $\homology_{i+1}$ and $\homology_{\ell-g-(i+1)}$ are
Cohen-Macaulay.  The inductive hypothesis implies that $E_{2}^{p,q}=0$
whenever $q>g$ and either $p\leq i$ or $p\geq \ell-g-i$.

Serre's condition $S_h$ on the homology modules and
Proposition~\ref{prp:S2} imply that the modules $\homology_k$ and
$\Ext{g}{\homology_{\ell-g-k}}{R}$ are isomorphic via the edge
homomorphism for any $k\geq 0$. It follows from Lemma~\ref{arrows} that 
the differentials emanating from row $g$ on any page of the spectral
sequence are zero maps. Hence, the top two possibly nonzero modules
$E_{r}^{p,q}$ in column $\ell-g-(i+1)$, namely those with
$(p,q)=(\ell-g-i-1,d-h)$ and $(\ell-g-i-1,d-h-1)$ (or just the one
module $E_{r}^{p,q}$ with $(p,q)=(\ell-g-i-1,d-h)$ if $d-h-1=g$),
have all differentials coming into or out of them equal to zero on
any page of the spectral sequence. Therefore, for these values
of $(p,q)$ one has $E_{2}^{p,q}=E_{3}^{p,q}=\cdots=E_{\infty}^{p,q}$.
On the other hand, Lemma~\ref{arrows} also yields that
$E_{\infty}^{p,q}=0$ for any $q>g$. So one gets that
\begin{align*}
\Ext{d-h}{\homology_{\ell-g-i-1}}{R} = 0= \Ext{d-h-1}{\homology_{\ell-g-i-1}}{R}
\end{align*}
(or if $d-h-1=g$ just that $\Ext{d-h-1}{\homology_{\ell-g-i-1}}{R}=0$ 
and so $\homology_{\ell-g-i-1}$ is clearly Cohen-Macaulay). 
This yields the improved inequality $\depth
\homology_{\ell-g-i-1} \geq h+2$ and therefore 
\begin{align*}
  \depth \homology_{i+1} + \depth \homology_{\ell-g-i-1} \geq 2h+2
  \geq \dim (R/({\bf y})) +2.
\end{align*}
holds. 
Now consider the Koszul homologies as modules over the ring $S=R/({\bf x})$ 
where ${\bf x}$ is an $R$-regular sequence of length $g$ in the ideal $({\bf y})$,  
and note that that the supports of $R/({\bf y})$ and $R/({\bf x})$ are the same. 
By the remark below applied to $S$, the inequality above and the earlier 
assumption that $({\bf y})$ is strongly Cohen-Macaulay imply that 
$\homology_{i+1}$ and $\homology_{\ell-g-i-1}$ are Cohen-Macaulay as desired.
\end{proof}

\begin{rmk}
\label{rmk:HO}
Let $S$ be a Cohen-Macaulay ring with canonical module $\omega_S$. 
If an $S$-module $M$ and its dual $M^\vee =  \Hom[S]{M}{\omega_S}$ satisfy 
\[
\depth M_\p + \depth (M^\vee)_\p \geq \dim S_\p + 2
\]
for all prime ideals $\p$ in $S$ of height at least 3, and if 
$M_\p$ is Cohen-Macaulay for all prime ideals $\p$ in $S$ of height at most 2, 
then $M$ is Cohen-Macaulay.  
This can essentially be found in the paper \cite{HO} of Hartshorne and Ogus, 
but we give the version as stated and proved by Huneke in \cite[Lemma~5.8]{Hu}.

Theorem~\ref{thm:dimR/2} can also be compared to this result for 
$M= \homology_i({\bf y})$ and $S=R/({\bf x})$ 
where ${\bf x}$ is an $R$-regular sequence of length $g$ in the ideal $({\bf y})$. 
In fact, by Proposition~\ref{prp:S2} and Remark~\ref{rmk:ExtgHom} 
\[
\homology_{\ell-g-i}({\bf y})\cong \Hom[R/({\bf x})]{\homology_i({\bf y})}{R/({\bf x})} \cong \homology_{i}({\bf y})^\vee.
\]
The hypotheses of the theorem give the inequalities 
\[
\depth \homology_{i}({\bf y})_\p + \depth (\homology_{i}({\bf y})^\vee)_\p \geq \dim (R/({\bf x}))_\p 
\]
for all $i$ and still yield that all $\homology_i({\bf y})$ are Cohen-Macaulay. 
Notice that by requiring high enough depth for all the Koszul homologies, 
we can weaken hypotheses on the depth required to get a conclusion of Cohen-Macaulay. 
\end{rmk}

\section{Sliding depth}
\label{sec:SD}

The main result of the section is a generalization of
\cite[Theorem~1.4]{HVV}, in which the authors show that ideals satisfying 
sliding depth are strongly Cohen-Macaulay in the presence of  conditions 
limiting the number of local generators of the ideal.

\begin{dfn}
An ideal $I$ of $R$ is {\em strongly Cohen-Macaulay} if all the
non-vanishing Koszul homology modules are Cohen-Macaulay.
\end{dfn}

\begin{dfn}
\label{def:SDh}
An ideal $I$ is said to satisfy {\em $h$-sliding depth} or $SD_h$ if it satisfies
\begin{align*}
\depth \homology_{i}(I) \geq  \min\{ d-g,d-\ell+i+h \}.
\end{align*}
\end{dfn}

The condition $SD_h$ localizes, as was shown in \cite{HSV}, where it was introduced. 

\begin{rmk}
\label{depthZ}
Let ${\bf y}$ a sequence of $\ell$ elements, and set $\koszul_i=\koszul_i({\bf y})$ 
and $\homology_i=\homology_i({\bf y})$.  Let $Z_i$ and $B_i$ be the submodules
of $\koszul_i$ of cycles and boundaries, respectively. 
From the exact sequences
\begin{align*}
0 \to Z_{i+1} \to &\koszul_{i+1} \to B_{i} \to 0 \quad \text{and}\\
0 \to B_{i} \to& Z_{i} \to \homology_{i} \to 0
\end{align*}
one can easily verify that the following inequalities hold for all $i\geq 0$
\begin{align*}
\depth Z_i \geq \min \{ d,d-\ell+i+h+1  \}
\end{align*}
when $I$ satisfies $SD_h$. 
\end{rmk}

Given an ideal $I$, denote by $\mu(I)$ its minimal number of generators.
The following theorem is a generalization of the main result in
\cite[Theorem~1.4]{HVV}. Notice that for $j\geq 0$ condition $(a)$
below is equivalent to $\mu(I_{\p})\leq \max\{\height{I},
\height{\p}+j\}$; with this formulation the theorem in \cite{HVV}
corresponds to the case $j=-1$ in the theorem below (their assumption
of Cohen-Macaulayness for $R/I$ is not in fact necessary; see Remark~\ref{optimal}). 
In fact, the proof that we give works for the case $j=-1$ as well, 
and hence provides a shorter proof of their result.

\begin{thm}
\label{thm:SDh}
Let $R$ be a local Gorenstein ring, and let $I$ be an ideal of $R$.
Fix an integer $j \geq 0$.
Suppose that the following conditions hold
\begin{itemize}
\item[(a)] $\mu(I_{\p}) \leq \height{\p}+j$ for all prime ideals $\p \supseteq I$.
\item[(b)] $I$ satisfies $SD_h$ where $h=\lceil \frac{j+1}{2} \rceil$.
\item[(c)] $\homology_i(I)$ satisfies Serre's condition $S_2$ as an
  $R/I$-module for $i=0,\dots,\ell-g$.
\end{itemize}
\noindent Then $I$ is strongly Cohen-Macaulay.
\end{thm}

\begin{rmk}
\label{optimal}
Notice that conditions (a) and (b) give that the homology modules
$\homology_i(I)$ automatically satisfy Serre's condition $S_2$ for
all $i\geq \lfloor \frac{j+1}{2} \rfloor +1$. Indeed, by \cite[\S
6(ii)]{HVV} the condition $SD_h$ localizes. This implies that
\[
\depth \homology_i(I_\p) \geq \dim R_\p - \mu(I_\p) + i + h \geq
-j+i+h = -h^\prime + i + 1
\]
for $h^\prime = \lfloor \frac{j+1}{2} \rfloor $ since
$h^\prime+h=j+1$.  For example, for $j=0$ (and $j=-1$), one gets that
$\homology_i(I)$ satisfies $S_2$ for all $i>0$, and hence
Theorem~\ref{thm:SDh} immediately yields Corollary~\ref{SDhj=0} below.
Similarly, for $j=1$, only $\homology_0(I)$ and $\homology_1(I)$ do
not automatically satisfy $S_2$ from the sliding depth condition, and
hence one obtains Corollary~\ref{SDhj=1} below.
\end{rmk}

\begin{cor}
\label{SDhj=0}
Let $R$ be a local Gorenstein ring, and let $I$ be an ideal of $R$
such that $R/I$ satisfies Serre's condition $S_2$. Suppose the
following conditions hold
\begin{itemize}
\item[(a)] $\mu(I_{\p}) \leq \height{\p}$ for all prime ideals $\p \supseteq I$
and
\item[(b)] $I$ satisfies $SD_1$.
\end{itemize}
\noindent Then $I$ is strongly Cohen-Macaulay.
\end{cor}

\begin{cor}
\label{SDhj=1}
Let $R$ be a local Gorenstein ring, and let $I$ be an ideal of $R$ 
such that $R/I$ satisfies Serre's condition $S_2$. Suppose the following conditions hold
\begin{itemize}
\item[(a)] $\mu(I_{\p}) \leq \height{\p}+1$ for all prime ideals $\p \supseteq I$
and
\item[(b)] $I$ satisfies $SD_1$.
\item[(c)] $\homology_1(I)$ satisfies Serre's condition $S_2$ as an $R/I$-module.
\end{itemize}
\noindent Then $I$ is strongly Cohen-Macaulay.
\end{cor}
We first give a quick proof of Theorem~\ref{thm:SDh} using the
spectral sequence from Construction~\ref{rmk:spectralsequence}. Since
there is an elementary proof modeled on that of \cite[Theorem~1.4]{HVV} 
that avoids spectral sequences, we also include it afterwards.

\begin{prf*}
Denote by $\homology_i$ the homology modules $\homology_i(I)$.
First note that if $\dim R/I \leq 2$ then the assertion follows
trivially from condition (c). Assume that $\dim R/I \geq 3$. Let
$\ell$ be the minimal number of generators of $I$ and $g$ be the
grade of $I$; by adding a set of indeterminates to $R$ and to $I$
one may assume that $g > \ell-g+h+1$.  Furthermore, since the
conditions on $I$ localize, one may assume by induction that $I$ is strongly
Cohen-Macaulay on the punctured spectrum of $R$. 
Consider the spectral sequence
\begin{align*}
E_{2}^{p,q}=\Ext{q}{\homology_p}{R} \Rightarrow \homology_{\ell-(p+q)}
\end{align*}
from Construction~\ref{rmk:spectralsequence}.  As before, one has that
$E_{2}^{p,q}=0$ for $q<g$.

We use induction on $i$ to prove that $\homology_{i}$ and
$\homology_{\ell-g-i}$ are Cohen-Macaulay. Note that for all $i\leq h$, by
assumption (b), the module $\homology_{\ell-g-i}$ is Cohen-Macaulay and
hence $\homology_{i}$ is also Cohen-Macaulay as one has $\homology_{i}(I) \cong
\Ext{g}{\homology_{\ell-g-i}(I)}{R}$ by 
Proposition~\ref{prp:S2}. We use induction to prove that
$\homology_i$ is Cohen-Macaulay for the remaining range $h < i <
\ell-g-h$, considering the base cases ($i\leq h$) done.  For the
inductive step, we suppose that $\homology_{p}$ and
$\homology_{\ell-g-p}$ are Cohen-Macaulay for all $p< i$ and we
show that $\homology_{i}$ and $\homology_{\ell-g-i}$ are
Cohen-Macaulay.  The sliding depth hypothesis $SD_h$ yields
inequalities
\begin{align*}
  \depth \homology_{i} \geq d - \ell + i + h, \quad \text{and} \quad \depth
  \homology_{\ell-g-i} \geq d-\ell + (\ell - g- i) + h.
\end{align*}
If they are not Cohen-Macaulay, we improve the inequality for $\depth
\homology_{\ell-g-i}$ by one in order to be able to apply Remark~\ref{rmk:HO}. 

The inductive hypothesis implies that $E_{2}^{p,q}=0$ whenever $q>g$
and either $p< i$ or $p> \ell-g-i$. In addition, by
Proposition~\ref{prp:S2} and by Lemma~\ref{arrows} the differentials
emanating from row $g$ on any page of the spectral sequence are zero
maps.  Therefore, all differentials landing in the term
$E_{r}^{p_0,q_0}$ with $p_0=\ell-g-i$ and $q_0=d-(d-g-i+h)=g+i-h$ are
zero.

Moreover, the hypothesis of $SD_h$ implies that $E_{r}^{p,q}=0$ for
$q>d-(d-g-p+h)$, that is, they vanish above the line though
$E_{r}^{p_0,q_0}$ with a slope of $-1$. Therefore, all differentials
emanating from $E_{r}^{p_0,q_0}$ are zero as well, and so one sees that
$E_{2}^{p_0,q_0} = E_{3}^{p_0,q_0} = \cdots = E_{\infty}^{p_0,q_0}$. But
Proposition~\ref{prp:S2} and Lemma~\ref{arrows} yield that
$E_{\infty}^{p_0,q_0}=0$; therefore one gets that $E_{2}^{p_0,q_0}=0$,
which yields the improved depth inequality $ \depth
\homology_{\ell-g-i-1} \geq d-\ell + (\ell - g- i) + h +1$.

In conclusion, one has 
\begin{align*}
\depth \homology_{\ell-g-i} + \depth \homology_{i}   &\geq    d-\ell + (\ell - g- i) + h +1 + d - \ell + i + h \\
&\geq (d-g) + (d+j-\ell) +2 \geq \dim R/I +2,
\end{align*}
where the last inequality follows from the fact that $d+j-\ell \geq 0$: 
Indeed condition (a) at $\p=\m$ yields the inequality $\ell \leq d+j$; 
note that even if $j=-1$ this inequality would hold as $\dim R/I \ge 3$.

Together with the earlier assumption that $I$ is strongly
Cohen-Macaulay on the punctured spectrum, this inequality now implies
that $\homology_{i}$ and $\homology_{\ell-g-i}$ are Cohen-Macaulay
as desired by the result \cite[Lemma~5.8]{Hu} or \cite{HO}, as described in Remark~\ref{rmk:HO},  
applied to the ring $S=R/({\bf x})$ 
where ${\bf x}$ is an $R$-regular sequence of length $g$ in the ideal $({\bf y})$,  
using the fact that the supports of $R/({\bf y})$ and $R/({\bf x})$ are the same.
\end{prf*}

The following is a proof of Theorem~\ref{thm:SDh} that avoids spectral sequences.

\begin{prf*}[Alternate Proof of Theorem~\ref{thm:SDh}]
  Let $\ell$ be the minimal number of generators of $I$ and $g$ be the
  grade of $I$.  As in the other proof, we may assume that one has
  $\dim R/I \geq 3$ and $g > \ell-g+h+1$ and that $I$ is strongly
  Cohen-Macaulay on the punctured spectrum of $R$.  
  
  Set $\homology_i = \homology_i(I)$ for all $i$.  Since $\homology_i$
  satisfies Serre's condition $S_2$ as an $R/I$-module for all $i$,
  there is an isomorphism $\homology_i\cong \Ext{g}{\homology_{\ell
      -g-i}}{R}$ by Proposition~\ref{prp:S2}. Next note by condition
  (b) that $\homology_{\ell-g-i}$, and hence $\homology_{i}$ by
  Proposition~\ref{prp:S2}, are Cohen-Macaulay for $i=0, \dots h$. In
  particular, one sees that $R/I$ is Cohen-Macaulay. In order to
  prove that $\homology_{i}$ and $\homology_{\ell-g-i}$ are
  Cohen-Macaulay for the remaining range $h<i<\ell-g-h$ we prove the
  inequality
\begin{align}
\label{HOinequality}
\depth \homology_{\ell-g-i} + \depth \homology_{i}   &\geq  \dim R/I +2
\end{align}
and the desired conclusion follows from Remark~\ref{rmk:HO} as in the other proof.

Since $I$ satisfies $SD_h$ one has inequalities
\begin{equation}
    \label{depthsum}
    \begin{split}
    \begin{aligned}
\depth \homology_{\ell-g-i} + \depth \homology_{i}   &\geq    d-\ell + (\ell - g- i) + h + d - \ell + i + h \\
&\geq (d-g) + (d+j-\ell) +1.
\end{aligned}
    \end{split}
  \end{equation}
Notice, as in the previous proof, that condition (a) at $\p=\m$ yields the inequality 
$\ell \leq d+j$ and so gives that $d+j-\ell \geq 0$.

If, in fact, $d+j-\ell \geq 1$, then for all $i$ one has that
\begin{align*}
  \depth \homology_{\ell-g-i} + \depth \homology_{i} &\geq d-\ell +2
  \geq d-g+2 = \dim R/I +2
\end{align*}
and hence $\homology_{i}$ is Cohen-Macaulay for all $i$.

If $d+j-\ell =0$, we use induction on $i$ to prove that
$\homology_{i}$ (equivalently, $\homology_{\ell-g-i}$) is
Cohen-Macaulay for all $i$, or, equivalently, that the
inequality~\ref{HOinequality} holds.  For $i=0$ the conclusion is
obvious as $R/I$ is Cohen-Macaulay. Note further that $\depth
Z_{\ell-g+1} \geq d-g+2$ since $Z_{\ell-g+1} \cong B_{\ell-g+1}$ and
$B_{\ell-g+1}$ has projective dimension at most $g-2$ as 
$\homology_{i}=0$ for all $i\geq \ell-g+1$.

For the inductive step, we prove a seemingly stronger statement. We
assume that $\homology_{i}$ is Cohen-Macaulay and that $\depth
Z_{\ell-g-(i-1)} \geq d-g+2$ and prove that $\homology_{i+1}$ is
Cohen-Macaulay and that $\depth Z_{\ell-g-i} \geq d-g+2$.

We use the following exact sequences from Remark~\ref{depthZ}
\begin{align}
\label{ZKB}
0 \to Z_{t+1} \to K_{t+1} \to B_{t} \to 0
\end{align}
\begin{align}
\label{BZH}
0 \to B_{t} \to Z_{t} \to \homology_{t} \to 0.
\end{align}

From the exact sequence (\ref{ZKB}) with $t=\ell-g-i$, depth-counting
yields that $\depth B_{\ell-g-i} \geq d-g+1$ from the inductive
hypothesis.  Thus the exact sequence (\ref{BZH}) with $t=\ell-g-i$
gives that
\begin{align*}
  \depth Z_{\ell-g-i} \geq \min \{ \depth B_{\ell-g-i}, \depth
  \homology_{\ell-g-i} \} = d-g,
\end{align*}
where equality holds since $\homology_{\ell-g-i}$ is Cohen-Macaulay.

The exact sequence (\ref{BZH}) yields the long exact sequence
\begin{equation}
    \label{BZHles}
    \begin{split}
      \begin{aligned}
0=&\Ext{g-1}{\homology_{\ell-g-i}}{R}  \to \Ext{g-1}{Z_{\ell-g-i}}{R} \to \Ext{g-1}{B_{\ell-g-i}}{R} \to\\
& \Ext{g}{\homology_{\ell-g-i}}{R} \to \Ext{g}{Z_{\ell-g-i}}{R} \to
\Ext{g}{B_{\ell-g-i}}{R}=0,
\end{aligned}
    \end{split}
  \end{equation}
where the last equality comes from the fact that $\depth B_{\ell-g-i}
\geq d-g+1$.  The exact sequences (\ref{ZKB}) and
(\ref{BZH}), together with the fact that $\Ext{j}{\homology_t}{R}=0$
for all $j<g$, yield a series of isomorphisms (note $g-1-i>0$ for
$i=0, \dots, \ell-g$)
\begin{align*}
\Ext{g-1}{B_{\ell-g-i}}{R} & \cong
\Ext{g-2}{Z_{\ell-g-i+1}}{R}  \\ &\cong
\Ext{g-2}{B_{\ell-g-i+1}}{R} \cong \cdots \   \cong
\Ext{g-1 -i}{B_{\ell-g}}{R}.
\end{align*}
One can see that the last Ext module is isomorphic to $\homology_{i}$
by noting that the tail of the Koszul complex is a free resolution of
$B_{\ell -g}$. Furthermore, the term $\Ext{g}{\homology_{\ell-g-i}}{R}$ is also isomorphic to
$\homology_{i}$. The exact sequence (\ref{BZHles}) then becomes
\begin{align*}
0 \to &\Ext{g-1}{Z_{\ell-g-i}}{R} \to \homology_{i} \xrightarrow{\psi} \homology_{i} \to \Ext{g}{Z_{\ell-g-i}}{R} \to 0.
\end{align*}
As the module $Z_{\ell-g-i}$ is a second syzygy, say of $W$, one has
that $\Ext{g}{Z_{\ell-g-i}}{R} \cong \Ext{g+2}{W}{R}$. But for primes
$\p$ of height $g$ or $g+1$ the module $\Ext{g+2}{W}{R}_\p$ vanishes
as $R_\p$ has injective dimension at most $g+1$.  So $\psi_\p$ is
surjective, hence an isomorphism for such primes; in particular,
$\psi$ is an isomorphism in codimension 1 over $R/I$. Since
$\homology_{i}$ satisfies $S_2$ as an $R/I$-module, this implies that $\psi$ is an
isomorphism. Therefore, $\Ext{g-1}{Z_{\ell-g-i}}{R}$ and
$\Ext{g}{Z_{\ell-g-i}}{R}$ vanish, yielding $\depth Z_{\ell-g-i} \geq
d-g+2$.

By the exact sequence (\ref{ZKB}) for $t=\ell-g-i-1$ a depth count
yields that $\depth B_{\ell-g-i-1} \geq d-g+1$. Similarly, the exact
sequence (\ref{BZH}) for $t=\ell-g-i-1$ yields
\begin{align*}
\depth \homology_{\ell-g-i-1} \geq \min \{ \depth B_{\ell-g-i-1}-1  , \depth Z_{\ell-g-i-1}  \} \geq d-g-i+h
\end{align*}
where the last inequality is due to Remark (\ref{depthZ}) and the fact that $i>h$.  Thus, the inequality (\ref{HOinequality}) holds with $i$ replaced by  $i+1$, completing the inductive step.
\end{prf*}

\section{Low dimension}
\label{sec:SD}

For low dimensional ideals, some surprising relationships between the Koszul homologies
fall out of the spectral sequence, as was also noted by Chardin in \cite{Ch}.

\begin{prp}
\label{prp:dim2}
Let $R$ be a local Gorenstein ring and let $I$ be an ideal such that
$\dim R/I =2$.  Assume that the Koszul homology module
$\homology_{\ell-g-i-1}(I)$ is Cohen-Macaulay for some $i$.  
If either $\homology_i(I)$ or $\homology_{\ell-g-i-2}(I)$ has positive depth, 
then $\homology_i(I)$ is Cohen-Macaulay.
\end{prp}

\begin{prf*}
Consider the spectral sequence from
Construction~\ref{rmk:spectralsequence}.  
Set $\homology_p = \homology_p(I)$ for all $p$. 
Since $\dim R/I=2=d-g$, the second
page $E_2$ has at most 3 nonzero rows, namely rows $g$,
$g+1$, and $g+2$. 
Thus, one has $E^{i,g+1}_2=E^{i,g+1}_{\infty}$. 
Since
$\homology_{\ell-g-i-1}$ is Cohen-Macaulay, then the edge map 
$\homology_{\ell-g-i-1} \to \Ext{g}{\homology_{i+1}}{R}$ 
is an isomorphism by Proposition~\ref{prp:S2}. 
One then has that 
$\Ext{d-1}{\homology_i}{R} = E^{i,g+1}_2 = E^{i,g+1}_{\infty} = 0$, 
where the last equality is by Lemma~\ref{arrows}. 
This implies that $\depth \homology_i \not = 1$. 
Therefore, if $\homology_i(I)$ has positive depth, then $\homology_i$ is Cohen-Macaulay.

Moreover, Lemma~\ref{arrows} also yields that the differential $d_2^{i+1,g}$ is zero.
Thus, the equality $ E^{i,d}_2= E^{i,d}_{\infty}$ holds. Therefore, $\Ext{g+2}{\homology_i}{R} = E^{i,d}_2$
is isomorphic to a submodule of $\homology_{\ell-g-i-2}$. Since  $R$ is Gorenstein and $g+2=d$, 
the module $\Ext{g+2}{\homology_i}{R}$ has  finite length. If $\homology_{\ell-g-i-2}$  has positive 
depth then $\Ext{g+2}{\homology_i}{R}$ vanishes and so  $\homology_i(I)$ is Cohen-Macaulay.
\end{prf*}

Similar results can be deduced in other low dimensions. As an example, consider the following result. 

\begin{prp}
\label{prp:dim3}
Let $R$ be a local Gorenstein ring and let $I$ be an ideal such that
$\dim R/I =3$.  Assume that the Koszul homology modules
$\homology_{\ell-g-i-1}(I)$ and $\homology_{\ell-g-i-2}(I)$ satisfy Serre's $S_2$ condition for some $i$.  
If both $\homology_i(I)$ and $\homology_{i-1}(I)$ have positive depth, 
then $\homology_i(I)$ is Cohen-Macaulay.
\end{prp}

\begin{prf*}
Consider the spectral sequence from
Construction~\ref{rmk:spectralsequence}.  
Set $\homology_p = \homology_p(I)$ for all $p$. 
Since $\dim R/I=3=d-g$, the second
page $E_2$ has at most 4 nonzero rows, namely rows $g$,
$g+1$, $g+2$, and $g+3$. 
Therefore, the only possible nonzero differential coming in or out of $E^{i,g+1}_r$ on any page 
is the one to $E^{i-1,g+3}_2$ on page 2, but the fact that $\homology_{i-1}(I)$ 
has positive depth implies that $E^{i-1,g+3}_2=0$. 
Therefore, one has that $E^{i,g+1}_2 = E^{i,g+1}_\infty$. 
But since the edge map $\homology_{\ell-g-i-1} \to \Ext{g}{\homology_{i+1}}{R}$ is an 
isomorphism by Proposition~\ref{prp:S2},  
Lemma~\ref{arrows} implies that $E^{i,g+1}_\infty$ vanishes.  
Therefore, $\depth \homology_i \not = 2$. 
Lemma~\ref{arrows} also yields that the differential $d_2^{i+1,g}$ is zero. 
This being the only possible nonzero differential coming in or out of $E^{i,g+2}_2$, 
one gets that $E^{i,g+2}_2 = E^{i,g+2}_\infty$. 
However, since the edge map $\homology_{\ell-g-i-2} \to \Ext{g}{\homology_{i+2}}{R}$ is an 
isomorphism by Proposition~\ref{prp:S2},  
Lemma~\ref{arrows} implies that $E^{i,g+2}_\infty$ vanishes.  
Therefore, $\depth \homology_i \not = 1$. 
Since $\depth \homology_i \not = 0$ by hypothesis, 
$\homology_i(I)$ must be Cohen-Macaulay.
\end{prf*}

\section*{Acknowledgments}

\noindent The authors would like to thank Craig Huneke for
useful conversations regarding this paper, including comments 
which simplified the proof of Theorem~\ref{thm:genduality}.

\bibliographystyle{amsplain} %

\begin{thebibliography}{10}

\bibitem{AG}
L.~L.~Avramov and E.~S.~Golod, \emph{The homology of algebra of the {K}oszul
  complex of a local {G}orenstein ring}, Mat. Zametki \textbf{9} (1971),
  53--58. \MR{0279157 (43 \#4883)}

\bibitem{AH}
L.~L.~Avramov and J.~Herzog, \emph{The {K}oszul algebra of a
  codimension {$2$} embedding}, Math. Z. \textbf{175} (1980), no.~3, 249--260.
  \MR{602637 (82g:13011)}

\bibitem{bruns-herzog}
W.~Bruns and J.~Herzog, \emph{Cohen-{M}acaulay rings}, Cambridge
  Studies in Advanced Mathematics, vol.~39, Cambridge University Press,
  Cambridge, 1993. \MR{1251956 (95h:13020)}

\bibitem{Ch}
M.~Chardin, \emph{Regularity of ideals and their powers}, Institut de
  Mathematiques de Jussieu, Prepublication \textbf{364} (2004).

\bibitem{G}
E.~S.~Golod, \emph{On duality in the homology algebra of a {K}oszul complex},
  Fundam. Prikl. Mat. \textbf{9} (2003), no.~1, 77--81. \MR{2072620}

\bibitem{HO}
R.~Hartshorne and A.~Ogus, \emph{On the factoriality of local rings of
  small embedding codimension}, Comm. Algebra \textbf{1} (1974), 415--437.
  \MR{0347821 (50 \#322)}

\bibitem{HSV}
J.~Herzog, A.~Simis, and W.~V.~Vasconcelos, \emph{On the arithmetic and
  homology of algebras of linear type}, Trans. Amer. Math. Soc. \textbf{283}
  (1984), no.~2, 661--683. \MR{737891 (86a:13015)}

\bibitem{HVV}
J.~Herzog, W.~V.~Vasconcelos, and R.~Villareal, \emph{Ideals with sliding
  depth}, Nagoya Math. J. \textbf{99} (1985), 159--172. \MR{805087 (86k:13022)}

\bibitem{H}
J.~Herzog, \emph{Komplexe, {A}ufl\"osungen und {D}ualit\"at in der
  lokalen {A}lgebra}, Habilitationsschrift, Universit\"at Regensburg (1974).

\bibitem{hu82}
C.~Huneke, \emph{Linkage and the {K}oszul homology of ideals}, Amer. J.
  Math. \textbf{104} (1982), no.~5, 1043--1062. \MR{675309 (84f:13019)}

\bibitem{Hu}
\bysame, \emph{Strongly {C}ohen-{M}acaulay schemes and residual intersections},
  Trans. Amer. Math. Soc. \textbf{277} (1983), no.~2, 739--763. \MR{694386
  (84m:13023)}

\bibitem{Ku}
E.~Kunz, \emph{Almost complete intersections are not {G}orenstein rings}, J.
  Algebra \textbf{28} (1974), 111--115. \MR{0330158 (48 \#8496)}

\bibitem{SV}
A.~Simis and W.~V.~Vasconcelos, \emph{The syzygies of the conormal module},
  Amer. J. Math. \textbf{103} (1981), no.~2, 203--224. \MR{610474 (82i:13016)}

\bibitem{Va-80}
Giuseppe Valla, \emph{On the symmetric and {R}ees algebras of an ideal},
 Manuscripta Math. \textbf{30} (1980), no.~3, 239--255. \MR{557107
 (83b:14017)}
 
 \end{thebibliography}


\providecommand{\arxiv}[2][AC]{\mbox{\href{http://arxiv.org/abs/#2}{\sf
      arXiv:#2 [math.#1]}}}
\providecommand{\oldarxiv}[2][AC]{\mbox{\href{http://arxiv.org/abs/math/#2}{%
      \sf arXiv:math/#2
      [math.#1]}}}\providecommand{\MR}[1]{\mbox{\href{http://www.ams.org/mathscine%
      t-getitem?mr=#1}{#1}}}
\renewcommand{\MR}[1]{\mbox{\href{http://www.ams.org/mathscinet-getitem?mr=#%
      1}{#1}}} \providecommand{\bysame}{\leavevmode\hbox
  to3em{\hrulefill}\thinspace}
\providecommand{\MR}{\relax\ifhmode\unskip\space\fi MR }
\providecommand{\MRhref}[2]{%
  \href{http://www.ams.org/mathscinet-getitem?mr=#1}{#2} }
\providecommand{\href}[2]{#2}

\end{document}